\documentclass[a4paper,12pt]{amsart}

\usepackage{framed}

\def\Cal{\mathcal}

\newcommand{\na}{\nabla}
\newcommand{\tr}{\text{tr}}

\newcommand{\horiztract}[2]{(#1 ~ \mid ~ #2)}

\newcommand{\V}{\mbox{\sf P}}
\newcommand{\J}{\mbox{\sf J}}

\newcommand{\ce}{{\Cal E}}
\newcommand{\ct}{{\Cal T}}
\newcommand{\cT}{{\mathcal T}}

\newcommand{\rpl}                         
{\mbox{$
\begin{picture}(12.7,8)(-.5,-1)
\put(0,0.2){$+$}
\put(4.2,2.8){\oval(8,8)[r]}
\end{picture}$}}

\def\cB{\mathcal{B}}

\def\bg{\mbox{\boldmath $g$}}

\newcommand{\si}{\sigma}

\newcommand{\rank}{\operatorname{Rank}}

\newcommand{\End}{\operatorname{End}}

\newtheorem{theorem}{Theorem}[section]
\newtheorem{lemma}[theorem]{Lemma}
\newtheorem{proposition}[theorem]{Proposition}
\newtheorem{corollary}[theorem]{Corollary}

\theoremstyle{remark}
\newtheorem{remark}[theorem]{\rm\bf Remark}
\newtheorem*{definition*}{\rm\bf Definition}

\newcommand{\nn}[1]{(\ref{#1})}

\newcommand{\Ric}{\operatorname{Ric}}
\newcommand{\nd}{\nabla}

\def\sideremark#1{\ifvmode\leavevmode\fi\vadjust{\vbox to0pt{\vss
 \hbox to 0pt{\hskip\hsize\hskip1em
 \vbox{\hsize3cm\tiny\raggedright\pretolerance10000
  \noindent #1\hfill}\hss}\vbox to8pt{\vfil}\vss}}}%

\renewenvironment{leftbar}[1][\hsize]
{%
\MakeFramed{\hsize#1\advance\hsize-\width\FrameRestore}%
}
{\endMakeFramed}
\newenvironment{HM-adds}{\begin{leftbar}}{\end{leftbar}}

\author{A.\ Rod Gover \& Heather Macbeth}

\title{Detecting Einstein
geodesics: Einstein metrics in projective and conformal geometry}
  \dedicatory{Dedicated to Mike Eastwood on the occasion of his
    60${}^{\rm th}$ birthday} \thanks{ARG gratefully acknowledges
    support from the Royal Society of New Zealand via Marsden Grant
    10-UOA-113. HRM is grateful for the hospitality of the University
    of Auckland.}
\begin{document}

\address{ARG: Department of Mathematics\\
  The University of Auckland\\
  Private Bag 92019\\
  Auckland 1142\\
 New Zealand;
Mathematical Sciences Institute\\
Australian National University \\ ACT 0200} \email{r.gover@auckland.ac.nz}
\address{HM: Department of Mathematics\\
  Princeton University; Fine Hall, Washington Rd\\
  Princeton, NJ 08544\\
  USA} \email{macbeth@math.princeton.edu}


\begin{abstract}
Here we treat the problem: given a torsion-free connection do its
geodesics, as unparametrised curves, coincide with the geodesics of an
Einstein metric?  We find projective invariants such that the
vanishing of these is necessary for the existence of such a metric,
and in generic settings the vanishing of these is also sufficient.  We
also obtain results for the problem of metrisability (without the
Einstein condition): We show that the odd Chern type invariants of an
affine connection are projective invariants that obstruct the
existence of a projectively related Levi-Civita connection.  In
addition we discuss a concrete link between projective and conformal
geometry and the application of this to the projective-Einstein
problem.
\end{abstract}

\maketitle \pagestyle{myheadings} \markboth{Gover \& Macbeth}{Einstein
  metrics in projective and conformal geometry}

\section{Introduction}

 Suppose that $\nabla$ is an affine connection on a manifold of dimension at least 2, and consider its geodesics as unparametrised curves. The
 collection of all such curves is called a {\em projective structure}.  Two
 connections differing only by torsion share the same geodesics, so
 for our considerations there will be no loss of generality in
 assuming that $\nabla$ is torsion free.  The problem of whether these
 paths agree with the (unparametrised) geodesics of a
 pseudo-Riemannian metric is a classical problem that is also
 attracting significant recent interest
 \cite{BDE,EM,Liouville,Mikes,NurMet,Sinjukov}.  
 In general dimensions this is a difficult problem. However a striking
 simplification occurs if we ask a variant of the original question:

\begin{quote} Given a torsion-free connection $\nabla$,
do its geodesics, as unparametrised curves,
    coincide with the geodesics of an Einstein metric?\end{quote}

The surprising point is that it is easier to treat this new problem
directly without first investigating metrisability, and this problem
is the main subject of our article. We develop a general theory that
leads to ways to construct invariants with the property that they
depend only on a projective structure and such that their vanishing is
{\em necessary} for the projective structure to be compatible with an
Einstein Levi-Civita connection (in the sense of the stated
problem). In fact, in the generic case (where the notion of generic
will later be defined precisely, see Remark \ref{genericr}) the
vanishing of certain of these invariants is also {\em sufficient} for
the projective structure to be compatible with an Einstein metric; see
Section \ref{mainS}.

Although the Einstein metricity problem is our main aim, we also
obtain new results for the metricity problem, that is where the
Einstein condition is omitted.

Both problems have a serious physical motivation in dimension 4. The
observations of a part of spacetime may recover the path data of 
many bodies. If we assume that these bodies follow the paths of some
affine connection we may use the invariants to test the hypothesis
that the connection involved is the Levi-Civita connection of some
Einstein metric. See \cite{HL,HL2,Matv} for further discussion
and also related results. On the side of mathematics the ideas
surrounding these problems have the potential to contribute deeply to
understanding the links between pseudo-Riemannian geometry and
projective geometry, a theme with strong classical tradition see
e.g.\ \cite{Liouville,Mikes,KM} and references therein. The study of
Einstein metrics is especially important and as we shall show they
have a very special role in projective geometry.  In this article we
also describe in Section \ref{pun} a useful concrete link between
conformal geometry and projective geometry; in particular this link
led to several of the ideas used in the article. Finally the ideas
developed here can provide a template for similar problems in the
setting of, for example, $h$-projective geometry.

Our approach to the problems partly uses the projective tractor
calculus following its treatment in \cite{BEG,CGH,GN2}, see Section
\ref{pT}. The problem of projective metricity is controlled by a
projectively invariant overdetermined linear partial differential
equation \cite{Mikes,Sinjukov}.  The corresponding prolonged
differential system is studied by Eastwood and Matveev in \cite{EM}.
They show that the existence of a Levi-Civita connection in the
projective class of a projective manifold is equivalent to the
existence of a suitably nondegenerate section of a symmetric product
of the basic tractor bundle that has the property that it is parallel
for a certain projectively invariant connection that they also
construct.  The connection involved is not the usual (i.e.\ normal)
tractor connection.  However in \cite{CGM}  the
authors show that the
section involved is parallel for the normal tractor connection if and
only if the solution corresponds to an Einstein metric (and there are
strong connections with the works \cite{CGHjlms,CGHpoly,CGH,Leitner} as
explained there).  This last fact is closely related to results of
Armstrong \cite{ArmstrongP1}, and it is this link between Einstein
metricity and parallel tractors that leads to a considerable
simplification, as the normal tractor connection is well
understood. In Theorem \ref{charth} we give a new direct proof of the
required fact, which extends a result of Armstrong.

Here are the main results in order of appearance.  In Theorem
\ref{chern1} we show that the Chern type curvature forms of a manifold
equipped with scale connection are projectively invariant. (A scale
connection is simply a torsion-free affine connection with vanishing
first Chern form, and there is always such a connection in a
projective class; see Proposition \ref{defsc} in Section \ref{pdense}.) 
This parallels the result
of Chern-Simons that the same forms for a Levi-Civita connection are
conformally invariant \cite{Chern-Simons}. We show in Proposition
\ref{tchern} that they coincide with the curvature forms of the
projectively invariant tractor connection; this is one way to
understand the projective invariance.  In Corollary \ref{chernmet} we
conclude that the $k$-odd curvature forms obstruct the existence of a
Levi-Civita connection in the projective class. To the best of our
knowledge this is a new result for the question of projective
metrisability.

In Section \ref{mainS} we show that if $\nabla$ is the Levi-Civita
connection of an Einstein metric then its projective Cotton tensor (as
defined in Section \ref{projg}) vanishes.  A first observation is that
this solves the problem completely in dimension 2: see Remark
\ref{dim2}. For other dimensions this suggests using the condition of
projective-to-Cotton-flat; see Proposition \ref{c-space}. The key idea
in Section \ref{mainS} is that if an affine connection is projectively
related to a Cotton-flat affine connection then one can obtain a
natural formula for the change of connection required to achieve this
(this combines expression \nn{upsilon-def} with the constructions of
$D$ in Section \ref{natural}); at least this is the case if the
projective Weyl curvature satisfies very mild conditions of
genericity. That formula then leads to the construction, in Theorem
\ref{appmain2}, of sets of sharp obstructions to the (non-zero scalar
curvature) Einstein metricity problem -- a set of {\em sharp
  obstructions} means a collection of projective invariants the
vanishing of which is necessary and sufficient for the existence of an
Einstein Levi-Civita connection in the projective class.  Along the
way in our treatment we observe in Remark \ref{pcot} that we also
obtain projective invariants that sharply obstruct the existence of a Cotton
flat connection in the projective class. These invariants are made
natural by again using Theorem \ref{appmain2} to substitute for $D$.

In Section \ref{prol} we describe a very general principle that may be
used to proliferate further projective invariants which obstruct the
existence of Einstein Levi-Civita connections in the projective class.
Finally in Section \ref{confg} we show that when a projective class
includes the Levi-Civita connection of an Einstein metric, the
corresponding projective and conformal tractor connections are very
simply related. This observation motivated many of the earlier
constructions and is used to give alternative proofs of some of the
results. Included in this section is an observation that on a
projective 3-manifold the projective Weyl tensor itself provides a
sharp obstruction to the Einstein metricity problem; see Corollary
\ref{n3ob}.

\section{Projective differential geometry} \label{projg}

Projective structures were defined in the introduction. An equivalent
definition is as follows.  A {\em projective structure} $(M^n,p)$,
$n\geq 2$, is a smooth manifold equipped with an equivalence class $p$
of torsion-free affine connections. The class is characterised
by the fact that two connections $\nabla$ and $\widehat\nabla$ in $p$
have the same {\em path structure}, that is the same geodesics up to
parametrisation. Explicitly the transformation relating these connections on $TM$ and $T^*M$ are given by 
 \begin{equation} \label{ptrans}
 \widehat\nabla_a Y^b = \nabla_a Y^b +\Upsilon_a Y^b + \Upsilon_c Y^c 
 \delta^b_a, \quad \mbox{and} \quad \widehat\nabla_a u_b = \nabla_a u_b - \Upsilon_a u_b -\Upsilon_b u_a,
\end{equation}
where $\Upsilon$
is some smooth section of $ T^*M$.
In the setting of a projective structure $(M,p)$ any connection $\nabla\in p$
is called a {\em Weyl connection} or {\em Weyl structure} on $M$. 

\subsubsection{Curvature for Weyl structures}\label{wstC}

Given a connection $\nabla\in p$ the curvature (on $TM$) is defined as
usual by 
$$
(\nabla_a\nabla_b-\nabla_b\nabla_a)v^c= R_{ab}{}^c{}_dv^d.
$$ 
Considering the tensor decomposition of this we see that it can be
written uniquely as
\begin{equation}\label{pweyl}
R_{ab}{}^c{}_d=W_{ab}{}^c{}_d+ 2\delta^c_{[a}\V_{b]d}+\beta_{ab}\delta^c_d ,
\end{equation}
where the {\em projective Weyl tensor} $W_{ab}{}^c{}_d $ shares the
algebraic symmetries of $R$, but in addition it is completely trace
free, and $\beta_{ab}$ is skew. The tensor $\V_{ab}$ is called the
{\em projective Schouten tensor} and from the algebraic Bianchi
identity $R_{[ab}{}^c{}_{d]}=0$, one finds $ \beta_{ab}=-2\V_{[ab]}$.
The Ricci tensor is defined by 
\begin{equation}\label{pschouten}
\Ric_{ab}:=R_{ca}{}^c{}_b \quad \mbox{and so} \quad (n-1)\V_{ab}=
\Ric_{ab}+\beta_{ab}
\end{equation}
and we note this is not generally symmetric.
From the differential Bianchi identity we obtain the identity
$$
\nabla_c W_{ab}{}^c{}_d=(n-2)C_{dab},  
$$ 
where $C_{dab} := 2 \nabla_{[a}\V_{b]d}$ is called the {\em
  projective Cotton tensor}.

Under a change of connection, as in \nn{ptrans}, one computes that the
Weyl curvature $W_{ab}{}^c{}_d$ is unchanged. Thus it is an invariant
of the projective structure $(M,p)$. In dimension 2 this vanishes but
the Cotton tensor is projectively invariant.  On the other hand
\begin{equation}\label{pcurvtrans}
\V^{\widehat{\nabla}}_{ab}= \V^{\nabla}_{ab}-\nabla_a
\Upsilon_b+\Upsilon_a\Upsilon_b\quad \mbox{and} \quad
\beta^{\widehat{\nabla}}_{ab}=
\beta^{\nabla}_{ab}+2\nabla_{[a}\upsilon_{b]} .
\end{equation}

\subsection{Projective densities and their connections}\label{pdense}
On any smooth $n$-manifold $M$ the highest exterior power of the
tangent bundle $(\Lambda^nTM)$ is a line bundle. The square of any
real line bundle has a canonical positive orientation. It follows that its square $(\Lambda^nTM)^2$ has a
canonical positive orientation and we shall fix that as its
orientation.  
 We may forget the tensorial
structure of $(\Lambda^nTM)^2$ and view this purely as a line bundle.
For our subsequent discussion it is convenient to take the positive
$(2n+2)^{th}$ root of $(\Lambda^nTM)^2$ and we denote this $K$ or
$\ce(1)$. Then for $w\in \mathbb{R}$ we denote $K^w$ by $\ce
(w)$. Sections of $\ce (w)$ will be described as {\em projective
  densities} of weight $w$.  
Given any bundle $\mathcal{B}$ we
shall write $\mathcal{B}(w)$ as a shorthand notation for
$\mathcal{B}\otimes \ce(w)$.

Now we consider a projective manifold $(M,p)$.  Each connection
$\nabla\in p$ determines a connection, also denoted $\nabla$, on
$(\Lambda^nTM)^2$ and hence on its roots $\ce(w)$, $w\in \mathbb{R}$.
  For $\nabla\in p$
let us (temporarily) denote the connection induced on $\ce(1)$ by
$D^\nabla$, and write $F^\nabla$ for its curvature. 
(In fact $F^\nabla=\beta^\nabla$, see \cite{GN2}, however we shall not 
need that fact here.) 
It is easily verified that,
under the transformation \nn{ptrans}, $D^\nabla$
transforms according to
\begin{equation}\label{pch}
D^{\widehat{\nabla}}_a= D^\nabla_a + \Upsilon_a,
\end{equation}
where we view $\Upsilon_a$ as a multiplication operator.  Since the
connections on $\ce(1)$ form an affine space modelled on
$\Gamma(T^*M)$ it follows that by moving around in $p$ we can hit any
connection on $\ce(1)$, and conversely a choice of connection on
$\ce(1)$ determines a connection in $p$ (or this may be seen by an explicit formula, e.g.\ \cite{CGM}).
 Let us summarise.
\begin{proposition}\label{weyly}
On a projective manifold $(M,p)$ a choice of Weyl structure is the same as
a choice of connection on $\ce(1)$.
\end{proposition}
Thus in the setting of
projective geometry we drop the notation $D^\nabla$ and simply write
$\nabla$ for the connection on $\ce(1)$ equivalent to a Weyl 
structure $\nabla$.

\begin{proposition}\label{defsc} 
Since $\ce(1)$ is a trivial bundle, any chosen trivialisation
determines a flat connection on $\ce(1)$ in the obvious way. Such a
connection will be called a \underline{scale connection}. 
That is there is a
special class $s$ of connections in $p$: $\nabla\in s$ if and only if
it preserves a section of $\ce(1)$. 
Again using that $\ce(1)$ is a trivial bundle it
follows that $\nabla $ is a  scale connection if and only if $F^\nabla=0$.
\end{proposition}

It follows from \nn{pch} that 
\begin{equation}\label{curvt}
F^{\widehat{\nabla}}=F^\nabla+ d \Upsilon 
\end{equation}
 and hence it is clear that if $\nabla$ and $\widehat{\nabla}$ are
 both scale connections then $d \Upsilon=0$. In fact since from the
 definition of scale connections $\Upsilon$ is then exact, as a change
 of scale connection is equivalent to a change of trivialisation of
 the globally trivial bundle $\ce(1)$. (Scale connections are called
 {\em exact Weyl structures} in \cite{CSbook}.)

\subsection{Projective tractor calculus}\label{pT}

By the definition of a projective structure $(M,p)$, there is no
preferred connection on the tangent bundle to $M$. This potentially
impedes calculation and understanding.  However there is a canonical
connection, known as the tractor connection, on a rank $n+1$ bundle
that is closely related to $TM$. This is due independently to Cartan
and Thomas \cite{ Cartan,Thomas}, here we follow \cite{BEG,CGHjlms}
and the conventions there.  In an abstract index notation let us write
$\ce_A$ for $J^1\ce(1)$, the first jet prolongation of $\ce(1)$.
Canonically we have the jet exact sequence
\begin{equation}\label{euler}
0\to \ce_a(1)\stackrel{Z_A{}^a}{\to} \ce_A
\stackrel{X^A}{\to}\ce(1)\to 0,
\end{equation}
where we have written $X^A\in \Gamma \ce^A(1)$ for the jet projection,
and $Z_A{}^a$ for the map inserting $\ce_a(1)$; these are both
canonical \cite{palais}.  We write $\ce_A=\ce_a(1)\rpl \ce(1)$ to
summarise the composition structure in \nn{euler}.  As mentioned, any
connection $\nabla \in p$ determines a connection on $\ce(1)$.  On the
other hand a connection on a vector bundle is the same as a splitting
of its 1-jet prolongation. Thus, in particular, a choice of connection
on $\ce(1)$ is a splitting of \nn{euler}.  If the tractor $U_A$ splits,
with respect to the connection $\nabla$, as $(\mu_i ~\mid~ \sigma)$,
then it splits with respect to a different $\widehat\nabla$ related to
$\nabla$ by \nn{ptrans} as 
\begin{equation}\label{ttrans}
(\mu_i+\Upsilon_i\sigma ~\mid~ \sigma).
\end{equation}

With respect to the direct sum decomposition $\ce_A
\stackrel{\nabla}{=} \ce_a(1)\oplus \ce(1) $ from a choice of
splitting $\nabla$, we define a connection $\nabla^{\mathcal{T}^*}_a$
on $\cT^*$ by,
\begin{equation}\label{pconn}
\nabla^{\mathcal{T}^*}_a (\mu_b ~ \mid ~ \si) := (\nabla_a \mu_b +
\V_{ab} \si ~ \mid ~ \nabla_a \si -\mu_a).
\end{equation}
It turns out that \nn{pconn} is
independent of the choice $\nabla \in p$, and so
$\nabla^{\mathcal{T}^*}$ is determined canonically by the projective
structure $p$. This is the {\em cotractor connection} as defined in
\cite{BEG} and is equivalent to the normal Cartan connection for the
Cartan structure of type $(G,P)$, see \cite{CapGoTAMS}. Thus we shall
also term $\ce_A$ the {\em cotractor bundle}, and we note the dual
bundle, called {\em tractor bundle} and denoted $\ce^A$ (or in index
free notation $\mathcal{T}$), has canonically the dual {\em tractor
  connection}: in terms of a splitting dual to that above this is
given by
\begin{equation}\label{tconn}
\nabla^\cT_a \left( \begin{array}{c} \nu^b\\
\rho
\end{array}\right) =
\left( \begin{array}{c} \nabla_a\nu^b + \rho \delta^b_a\\
\nabla_a \rho - \V_{ab}\nu^b
\end{array}\right).
\end{equation}

In the following we shall normally use simply $\nabla$ to denote
$\nabla^\cT$, $\nabla^{\mathcal{T}^*}$, or any of the connections
these induce on tensors products and powers of $\cT$ and $\cT^*$.  We
may use the same notation for a connection in $p$ and the coupling of
this with the tractor connection, but the meaning will clear by
context.

In the next Sections we will need the curvature $\Omega$ of the
projective tractor connection $\na$. This is defined by
$(\na_a\na_b-\na_b\na_a) U^C=\Omega_{ab}{}^C{}_D U^D$. This is easily
computed in terms of the curvature of any Weyl connection in the
projective class. First we need  some notation.

Let us denote by 
$$
Y_A:\ce(1)\to \ce_A\quad \mbox{and}\quad  Y^A{}_a: \ce_A\to \ce_a(1)
$$ the bundle maps splitting \nn{euler} as determined by (and
equivalent to) some connection $\nabla$ on $\ce(1)$.  So we have $X^AY_A=1$, $Z_A{}^a Y^A{}_b=\delta^a_b$
(the section of $\End(TM)$ that is the identity at every point) and
all other tractor index contractions of a pair from $X^A$, $Y_A$,
$Z_A{}^a$, $Y^A{}_b$ results in zero. Thus, for example, if
$\Gamma(\ce_A) \ni U_A \stackrel{\nabla}{=} (\mu_b ~ \mid ~ \si)\in
\Gamma(\ce_a(1)\oplus \ce(1)) $ then this means
$$
U_A=Z_A{}^a\mu_a+ Y_A \si .
$$

In this notation the curvature of the tractor connection 
is easily calculated to be (cf.\ \cite{BEG})
\begin{equation}\label{tcurvform}
\Omega_{ab}{}^C{}_D = W_{ab}{}^c{}_dZ_D{}^dY^C{}_c - C_{dab}Z_D{}^dX^C .
\end{equation}
That is, schematically, for a tractor $U^A=\begin{pmatrix}\nu^a\\ \rho\end{pmatrix}$, 
\[
\Omega_{ab}{}^C{}_DU^D=
\begin{pmatrix}
W_{ab}{}^c{}_d & 0 \\
-C_{dab} & 0
\end{pmatrix}
\begin{pmatrix}
\nu^d \\ \rho
\end{pmatrix}.
\]
It follows that 
\begin{equation}\label{XW}
\Omega_{ab}{}^C{}_D X^D=0.
\end{equation}

\section{Characterising Einstein in projective geometry}\label{char}

We call a symmetric form $h^{AB}$ on $\cT^*$ a \emph{sub-metric}, if
the restriction of $h^{AB}$ to the invariant sub-bundle $\ce_a(1)$
(see \nn{euler}) is genuinely a metric, i.e.\ non-degenerate.
Sub-metrics that are parallel for the tractor connection have an
important interpretation.
\begin{theorem} \label{charth}
The sub-metrics on $\cT^*$, compatible with the tractor connection
$\nabla$, are in one-to-one correspondence with the Einstein metrics
on $M$ whose Levi-Civita connection is in $p$.
\end{theorem}

In the case when the Einstein metric is not Ricci-flat, the sub-metric
on $\cT^*$ is a metric and hence its inverse is a compatible metric on the
\emph{projective tractor bundle}, $\cT$.  This compatible
metric, together with the analogous correspondence theorem, has been
described by Armstrong \cite{ArmstrongP1}. 
On the other hand
\cite[Theorem 3.3]{CGM} establishes an equivalent of Theorem
\ref{charth} by characterising what is called a normal solution of a
certain overdetermined partial differential equation, each suitably
non-degenerate solution of which corresponds to a metric giving a
Levi-Civita connection in the projective class. The remainder of this
Section is used to give a simple direct proof of the Theorem using the
projective tractor calculus.

Recall that a connection $\nabla\in p$ is the same as a splitting of
the sequence \nn{euler}.
\begin{lemma}\label{lemon} Let $\nabla$ be some splitting  of \nn{euler}.
Suppose that $h^{BC}$ is a section of  $\text{Sym}^2(\cT)$ which is diagonal
with respect to the splitting $\nabla$,
\[
\begin{pmatrix}
g^{ab} & 0 \\
0 & \tau
\end{pmatrix} .
\]
Then in the splitting $\nabla$ the tractor connection on $h^{BC}$ is given by
$$\nabla_a h^{BC}=
\begin{pmatrix}
\na_ag^{bc} & \tau \delta_a{}^c-g^{bc}\V_{ab} \\
\tau \delta_a{}^c-g^{bc}\V_{ab} & \nabla_a \tau
\end{pmatrix}.$$
\end{lemma}
\begin{proof} This is immediate from \nn{tconn}.
\end{proof}

First suppose that $g_{ab}$ is Einstein, with
$\V_{ab}=\lambda g_{ab}$. We define a metric $h^{AB}$ on the cotractor
bundle as follows. We work in the splitting of $g$'s Levi-Civita connection
$\nabla$. For $U_A\stackrel{\nabla}{=}(\mu_a ~\mid~ \sigma)$,
$\underline{U}_A\stackrel{\nabla}{=}( \underline{\mu}_a ~\mid ~ \underline{\sigma} )$, $h^{AB}$ is defined by
\begin{equation}\label{gscale}
h^{AB}U_A\underline{U}_B=g^{ab}\mu_a\underline\mu_b + \lambda\sigma\underline\sigma.
\end{equation}
From the fact that $\nabla$ is a Levi-Civita connection for an
Einstein metric $g$ we have $\nabla_cg^{ab}=0$, $\nabla_c\lambda=0$,
and $\lambda g_{ab}=\V_{ab}$. So by Lemma \ref{lemon} $h^{AB}$ is
compatible with the tractor connection.

Conversely, suppose that $p$ is a projective equivalence class such
that there exists a sub-metric $h^{AB}$ on $\cT^*$, compatible with
the tractor connection $\na$.  The restriction of $h^{AB}$ to
$\cT^*$'s invariant sub-bundle $\ce_a(1)$ is a metric $g^{ab}$. This
induces a metric on
\[
\big(\Lambda^n[T^*M(1)]\big)^2=\big(\Lambda^nT^*M\otimes \ce(1)^n\big)^2
=\ce(-2n-2)\otimes \ce(2n) = \ce(-2),
\]
 hence also a metric on
$\ce_a=\ce_a(1)\otimes\ce(-1)$. Thus we obtain a metric on all
weighted tensor bundles and this is compatible with trivialising the
density bundles using positively oriented unit length sections.  We
will use the notation $g$ to mean any of these metrics, with the
meaning clear from context.
\begin{lemma}\label{dia}
There is a unique connection $\na$ in $p$ which
  splits the Euler sequence
$$
0\to \ce_a(1)\to \cT^*\to \ce(1)\to 0
$$
in agreement with the orthogonal splitting via $h$.
\end{lemma}

\begin{proof}
Pick some $\nabla\in p$. By the dual of \nn{ttrans} a projective change of
$\nabla\mapsto\nabla'$ by a 1-form $\Upsilon_i$ induces the change of
splitting
\[
\begin{pmatrix}
(g')^{ab} & (v')^a \\
(v')^b & \tau'
\end{pmatrix}
=
\begin{pmatrix}
g^{ab} & v^a - g^{ab}\Upsilon_b \\
v^b - g^{ab}\Upsilon_a & \tau - 2 v^a\Upsilon_a + g^{ab}\Upsilon_a\Upsilon_b
\end{pmatrix}
\]
of the bundle $\text{Sym}^2(\cT)$.  (Here $g^{ab}$, $v^a$, $\tau$
denote sections of $\ce^{ab}(-2)$, $\ce^a(-2)$, $\ce(-2)$.) Since
$g^{ab}$ is nondegenerate, there is thus a unique change of connection
$\Upsilon_b$ such that $(v')^a=v^a-g^{ab}\Upsilon_b$ vanishes.
\end{proof}

With respect to the splitting $\na$ determined by Lemma \ref{dia} the
metric $h$ is given on cotractors
\[
U_A=(\mu_a  ~\mid ~ \sigma), \quad
\underline{U}_A=( \underline{\mu}_a ~\mid ~ \underline{\sigma}),
\]
by
\[ 
h^{AB}U_A\underline{U}_B=
g^{ab}\mu_a\underline{\mu}_b+\tau\sigma\underline{\sigma}.
\]

Since $\na$ and $h^{BC}$ are compatible, by Lemma \ref{lemon} we
therefore have that $\nabla_cg^{ab}=0$, $\nabla_c \tau =0$, and
$\tau g_{ab}=\V_{ab}$.  Using these identities in turn we conclude first
(from $\nabla_cg^{ab}=0$) that $\na$ is $g$'s Levi-Civita connection,
next (using $\nabla \tau=0$) that $\tau $ has constant length with respect to
$g$, and hence finally (from $\tau g_{ab}=\V_{ab}$) that $g$ is Einstein.

\section{Obstructions to projectively metric connections}

Here we construct projective invariants the vanishing of which is necessary for 
a projective class to include a Levi-Civita connection. 

\subsection{The obstructions}

 Here we work on an arbitrary manifold
with affine connection $(M,\nabla)$ except with symmetric Schouten
tensor, in other words $\nabla$ is an arbitrary scale connection.
Recall the decomposition \nn{pweyl} of the full curvature tensor as
$$ R_{ab}{}^c{}_d = W_{ab}{}^c{}_d + \delta_a^c\V_{bd}-
\delta_b^c\V_{ad},
$$
where $W$ denotes the projective Weyl curvature.
From this we obtain the following theorem.
\begin{theorem}\label{chern1}
Consider an manifold $M$ equipped with a scale connection $\nabla$, and  for each $k\in
\mathbb{Z}_{>0}$
 define a 
curvature form $p_k$ on $(M,\nabla)$, by
$$
(p_k)_{a_1a_2\cdots a_{2k}}:=R_{a_1a_2}{}^{c_1}{}_{c_2}R_{a_3a_4}{}^{c_2}{}_{c_3}
\cdots R_{a_{2k-1}a_{2k}}{}^{c_k}{}_{c_1} ,
$$ where (here and below) the $a_1,\cdots ,a_{2k}$ are
skewed over.  Then
$$
(p_k)_{a_1a_2\cdots a_{2k}}=
W_{a_1a_2}{}^{c_1}{}_{c_2}W_{a_3a_4}{}^{c_2}{}_{c_3} \cdots W_{a_{2k-1}a_{2k}}{}^{c_k}{}_{c_1},
$$
and so $p_k$ depends only on the projective class $[\nabla]$ of
$\nabla$.
\end{theorem}

\begin{proof}
The other terms that arise when calculating $p_k$ all contain multiplicands of at least one of the following two forms:
\begin{eqnarray*}
W_{a_{2i-1}a_{2i}}{}^{c_i}{}_{c_{i+1}} \delta_{[a_{2i+1}}{}^{c_{i+1}} \V_{a_{2i+2}]c_{i+2}}
&=&W_{a_{2i-1}a_{2i}}{}^{c_i}{}_{[a_{2i+1}} \V_{a_{2i+2}]c_{i+2}}\\
\delta_{[a_{2i-1}}{}^{c_i}\V_{a_{2i}]c_{i+1}} \delta_{[a_{2i+1}}{}^{c_{i+1}} \V_{a_{2i+2}]c_{i+2}}
&=&
\delta_{[a_{2i-1}}{}^{c_i}\V_{a_{2i}][a_{2i+1}} \V_{a_{2i+2}]c_{i+2}}
\end{eqnarray*}
When the full skew over $a_i$'s is taken, both these must vanish:  the former because of the symmetry $0=W_{[ij}{}^k{}_{l]}$ (the skew over all three bottom indices), the latter because $0=\V_{[ij]}$.
\end{proof}

Thus $p_k$ is a well-defined natural and canonical form on any
projective manifold $(M,p)$.  For $k$ odd $p_k$ has an interpretation: it
obstructs the existence of a Levi-Civita connection in $p$.
\begin{corollary} \label{chernmet}
Let $(M,p)$ be a projective manifold.
For odd positive integers $k$, if $p_k\neq 0$ then there is no
Levi-Civita connection in $p$.
\end{corollary}

\noindent{\bf Proof.}  Suppose that there is a Levi-Civita connection
$\nabla \in p$. Since $p_k$ is projectively invariant we may calculate
$p_k$ using this Levi-Civita connection $\nabla$, which we note is a
scale connection.  We have
$$
p_k=R_{a_1a_2}{}^{c_1}{}_{c_2}R_{a_3a_4}{}^{c_2}{}_{c_3}
\cdots R_{a_{2k-1}a_{2k}}{}^{c_k}{}_{c_1}
$$
where $R_{ab}{}^c{}_d$ is the (Riemann) curvature of $\nabla$. But
for $k$ odd this means $p_k=0$, as $R_{abcd}=-R_{abdc}$.  \quad $\Box$

For completeness, we establish that the above corollary is not vacuous:

\begin{proposition}
For each odd $k\geq 3$, there exist projective equivalence classes
$p$, such that the curvature form $p_k$ does not uniformly vanish.
\end{proposition}

\begin{proof}
Let $\mathcal{V}$ be the sub-bundle of $\Lambda^2T^*M\otimes End(TM)$
consisting of tensors $A_{ab}{}^c{}_d$ which are tracefree and satisfy
the Bianchi identity:
$$
A_{ab}{}^c{}_c=A_{ac}{}^c{}_b=0, \quad A_{ab}{}^c{}_d+A_{bd}{}^c{}_a+A_{da}{}^c{}_b=0.
$$ 
Note that if a connection $\nabla$ is torsion-free with symmetric
Schouten tensor, then its projective Weyl curvature $W$ is a section
of $\mathcal{V}$.

Let $x\in M$.  First, we show that every element of the vector space $\mathcal{V}_x$ is the projective Weyl curvature at $x$, $W|_x$ say, of some torsion-free connection $\nabla$ with symmetric Schouten tensor.  Indeed, let $\nabla$ be the connection whose Christoffel symbols in fixed co-ordinates $(x^i)$ centred at $x$ are
$$
\Gamma_{jk}^i=\frac{1}{3}S_{aj}{}^i{}_kx^a,
$$
where $S_{ij}{}^k{}_l=2A_{i(j}{}^k{}_{l)}$.  Since $\Gamma_{jk}^i$ is symmetric in $j,k$, the connection $\nabla$ is torsion-free.  The curvature of $\nabla$ is
\begin{eqnarray*}
R_{ij}{}^k{}_l&=&\partial_i\Gamma_{jl}^k-\partial_j\Gamma_{il}^k+\Gamma_{jl}^a\Gamma_{ia}^k-\Gamma_{il}^a\Gamma_{ja}^k\\
&=&\frac{1}{3}\left[S_{ij}{}^k{}_l-S_{ji}{}^k{}_l+\left(S_{bj}{}^a{}_lS_{ci}{}^k{}_a - S_{bi}{}^a{}_lS_{cj}{}^k{}_a\right)x^bx^c\right]\\
&=&A_{ij}{}^k{}_l+\frac{1}{3}\left(S_{bj}{}^a{}_lS_{ci}{}^k{}_a - S_{bi}{}^a{}_lS_{cj}{}^k{}_a\right)x^bx^c.
\end{eqnarray*}
Since $A$, thus $S$, is totally tracefree, the Ricci curvature of $\nabla$ simplifies to
\begin{eqnarray*}
R_{jl}&=&-\frac{1}{3}S_{bk}{}^a{}_lS_{cj}{}^k{}_ax^bx^c\\
&=&-\frac{1}{3}S_{bk}{}^a{}_lS_{ca}{}^k{}_jx^bx^c,
\end{eqnarray*}
where the last step is by the $a,j$ symmetry of $S$.  This is clearly symmetric in $j$ and $l$, as required.  Finally, at the point $0=x$, we have
\begin{eqnarray*}
R_{ij}{}^k{}_l&=&A_{ij}{}^k{}_l\\
R_{jl}&=&0,
\end{eqnarray*}
so as required the projective Weyl curvature is $A_{ij}{}^k{}_l$.

Suppose now that $n\geq 2k+2$.  There then exist a form $\omega\in\Lambda^2T^*_xM$ and an endomorphism $B\in End(T_xM$, such that
$$
\omega^{\wedge k}\neq 0, \quad \tr(B^k)\neq 0, \quad \tr(B) = 0, \quad \omega(\cdot, B\cdot) = 0.
$$
The tensor $A\in\Lambda^2T^*_xM\otimes End(T_xM)$,
$$
A_{ab}{}^c{}_d:= 2 \omega_{ab}B_d{}^c-\omega_{da}B_b{}^c-\omega_{bd}B_a{}^c,
$$
belongs to the subspace $\mathcal{V}_x$.  Thus there exists a connection $\nabla$ whose projective Weyl curvature at $x$ is $A$.

Split $A$ as $A^{(0)}+A^{(1)}+A^{(2)}$, where
$$
(A^{(0)})_{ab}{}^c{}_d= 2 \omega_{ab}B_d{}^c,
\quad (A^{(1)})_{ab}{}^c{}_d=-\omega_{da}B_b{}^c,
\quad (A^{(2)})_{ab}{}^c{}_d=-\omega_{bd}B_a{}^c,
$$
 and observe that the assumption $\omega(\cdot, B\cdot) = 0$ makes
all traces starting with $A^{(1)}$ or $A^{(2)}$ in $p_k$'s
construction vanish:
\begin{eqnarray*}
(A^{(1)})_{a_1a_2}{}^c{}_r(A^{(i)})_{b_0b_1}{}^r{}_{b_2}
&\sim\ \omega_{ra_1}B_{a_2}{}^c\omega_{b_{i}b_{i+1}}B_{b_{i+2}}{}^r&= \ 0,\\
(A^{(2)})_{a_1a_2}{}^c{}_r(A^{(i)})_{b_0b_1}{}^r{}_{b_2}
&\sim\ \omega_{a_2r}B_{a_1}{}^c\omega_{b_{i}b_{i+1}}B_{b_{i+2}}{}^r&= \ 0.
\end{eqnarray*}

We thus have
\begin{eqnarray*}
(p_k)|_x&=&
\sum_{\sigma\in\Sigma_{2k}}(-1)^\sigma\prod_{i=1}^k(A^{(0)})_{a_{\sigma(2i-1)}a_{\sigma(2i)}}{}^{c_i}{}_{c_{i+1}} \\
 &=&\frac{2^k}{(2k)!}\omega^{\wedge k}\tr(A^k),
\end{eqnarray*}
which is nonzero.
\end{proof}

\begin{remark}\label{global}
For odd positive integers $k$ the $p_k$ are necessarily
exact. Although finding a primitive will involve making a choice, it
should be that this has some interesting geometric applications. For
example, for a given such primitive, on a cycle that is homologically
trivial its integral is independent of choice (obviously projectively
invariant) and is a global obstruction to the projective metricity of
the ambient space.
\end{remark}

\subsection{A tractor motivation}

On a projective manifold $(M,p)$ let us define the tractor connection
curvature forms
\begin{equation}\label{obs}
q_k:=\Omega_{a_1a_2}{}^{C_1}{}_{C_2}\Omega_{a_3a_4}{}^{C_2}{}_{C_3} \cdots \Omega_{a_{2k-1}a_{2k}}{}^{C_k}{}_{C_1},
\end{equation}
where the $a_1,\cdots ,a_{2k}$ are skewed over. Since the tractor
curvature $\Omega$ is determined by the projective structure, by
construction the $q_k$ are projectively invariant (i.e.\ determined by
$(M,p)$). Of course these are characteristic type forms for a
connection, and so are closed forms by the Bianchi identity (here on
the tractor curvature).  Since {\em as a vector bundle} the projective
tractor bundle is a trivial extension of the tangent bundle (twisted
by a trivial line bundle) it follows at once that, for each $k$, the $q_k$
must represent the same de Rham cohomology class as $p_k$ (and this is
of course trivial if $k$ is odd). In fact a much stronger result is
true, as follows.
\begin{proposition}\label{tchern}
Let $(M,p)$ be a projective manifold. Then
$q_k=p_k$.
\end{proposition}
\noindent{\bf Proof.}  Recall from \nn{tcurvform}, the Cotton term in
the tractor curvature $\Omega_{ab}{}^C{}_D$ has coefficient $X^C$, but from \nn{XW} we have
$\Omega_{ab}{}^C{}_D X^D=0$. Thus using the identity $Z_A{}^aY^A{}_b=\delta^a_b$ the result follows from Theorem \ref{chern1}. 
\quad $\Box$ 

\reversemarginpar

This result puts the curvature forms $p_k$ in a nice setting. As
defined originally they are the curvature forms of affine connection,
and so not obviously projectively invariant. In Theorem \ref{chern1}
they are seen to arise from contractions of the projective Weyl
curvature so the projective invariance is then clear, but in that
formulation they are not manifestly from a connection. On the other
hand by the Proposition \ref{tchern} here we see they do arise as
characteristic type polynomials from the curvature of a projectively
invariant connection. The situation here is parallel to a number of
conformal results: When $\nabla $ is a Levi-Civita connection then it
is an important result of Chern-Simons that the $p_k$ are conformally
invariant \cite{Chern-Simons} (see also \cite{avez}). In that case
they can be viewed to have a conformal tractor origin \cite{BrGoPont}.

This tractor interpretation of the $p_k$'s was our
initial motivation for their study.  That the odd $p_k$'s are
obstructions to projective-metricity, is less surprising when it is
observed that 
they are at least obstructions to projective-Einstein-metricity.  This argument
goes as follows:

  If $\na$ preserves a metric $h$ on $\cT$ then the curvature
  2-form $\Omega$ necessarily takes values in the
  $\mathfrak{so}(h)$ subbundle of $\operatorname{End} (\cT)$. That means
  exactly that,
\begin{equation}\label{skew}
\Omega_{abCD}=-\Omega_{abDC} ,
\end{equation}
where we used $h$ to identify $\cT$ with $\cT^*$ and, in particular,
lower a tractor index of $\Omega_{ab}{}^C{}_D $.

It follows that when $k$ is odd the $q_k$ are projectively invariant
obstructions to the existence of a non-Ricci-flat Einstein metric
compatible with the projective structure. In fact it is not hard to strengthen the result by dropping
the ``non-Ricci-flat'' restriction by using also the results of
Section \ref{pun}. But we leave this as an exercise since, in any
case, the result in Corollary \ref{chernmet} supersedes this.

\section{Obstructions to projectively Einstein connections}\label{mainS}

Here we show and exploit the fact that if a Levi-Civita connection is
Einstein then its Cotton tensor must vanish. This idea is partly
inspired by the approach to analogous questions in conformal geometry
as developed in \cite{KNT,GN}.

\subsection{A projective C-space condition}\label{cspSec}

\begin{proposition}\label{c-space}
Let $\nabla$ be a torsion-free affine connection.
If $\nabla$ is projectively equivalent to some Einstein metric's
Levi-Civita connection, via (as in Equation \nn{ptrans}) the
1-form $\Upsilon_i$, then
\begin{equation}\label{ceqn}
C_{kij}+W_{ij}{}^l{}_k\Upsilon_l=0.
\end{equation}
\end{proposition}
\begin{remark}  \label{dim2}
 In dimension 2 the situation is rather special.  The projective Weyl
  tensor $W_{ij}{}^l{}_k$ is necessarily zero and this shows that the
  Cotton tensor is a projective invariant that obstructs the existence
  of an Einstein Levi-Civita connection in the projective class. In
  fact, as is well known, the Cotton tensor is a sharp obstruction to
  the existence of a flat connection in the projective class. Hence it
  is a sharp obstruction to the existence of an Einstein metric's
  Levi-Civita connection in the projective class.

 In
  all dimensions a connection $\na$ satisfies \nn{ceqn} if and only if
  $\nabla$ is projectively related to a connection with vanishing
  Cotton tensor, and this is a precondition to $\nabla$ being the Levi-Civita
  connection of an Einstein metric.
\end{remark}

In this section we give two proofs of this proposition: one conceptual
using the tractor language, and one by direct computation.

\begin{proof}[Tractor proof]
Let $p$ be a projective equivalence class of connections which
contains some Einstein metric's Levi-Civita connection.  First we
observe the existence of a distinguished rank-1 sub-bundle of the
cotractor bundle.
\begin{itemize}
\item In the case when the Einstein metric contained in $p$ is not
  Ricci-flat, this proceeds as follows.  By Theorem \ref{charth}, we
  have an induced metric $h^{AB}$ on the cotractor bundle,
  compatible with the tractor connection $\na$.  Inverting, this
  induces a metric $h_{AB}$ on the tractor bundle.  We then
  have a distinguished nonvanishing section of the weighted cotractor
  bundle $\ce_A(1)$, given by
\[
U_A := h_{AB}X^B.
\]

\item In the case when the Einstein metric is Ricci-flat, by Theorem
  \ref{charth}, we have an induced sub-metric $h^{AB}$ on the
  cotractor bundle, compatible with the tractor connection $\na$.
  The nullspace of the sub-metric $h^{AB}$ is everywhere of rank 1.
  This defines a rank-1 sub-bundle of the cotractor bundle which is
  preserved by the tractor connection.
\end{itemize}

We note that, for either case, it is immediate from \nn{gscale} that
in the splitting determined by an Einstein  Levi-Civita connection in $p$ the distinguished line sub-bundle consists
of the tractors
\[
V_A=\horiztract{0}{\sigma},
\]
where $\sigma$ is a section of $\ce(1)$.
Hence, by \nn{ttrans}, in an arbitrary
splitting $\na\in p$, the distinguished line sub-bundle consists of
the tractors
\[
V_A=\horiztract{-\sigma \Upsilon_i}{\sigma},
\]
where $\sigma$ is a section of $\ce(1)$, and $\Upsilon_i$ is the
1-form which defines the projective change to an Einstein metric's
Levi-Civita connection.

Next we observe that
 the rank-1 subbundle lies in the nullity of the cotractor curvature. 
\begin{itemize}
\item In the case when the Einstein metric is not Ricci-flat, note
  that using \nn{XW} we have
\[
0=\Omega_{ab}{}^C{}_DX^D=\Omega_{ab}{}^C{}_Dh^{DE}U_E.
\]
Since the tractor connection preserves $h^{DE}$, we have
\[
\Omega_{ab}{}^C{}_Dh^{DE}=-\Omega_{ab}{}^E{}_Dh^{DC}.
\]
Substituting back,
\[
0=\Omega_{ab}{}^C{}_D h^{DE}U_E=-\Omega_{ab}{}^E{}_D h^{DC}U_E.
\]
Since $h$ is nondegenerate we conclude that $\Omega_{ab}{}^E{}_DU_E=0$.

\item In the Ricci-flat case this 
follows from \nn{tcurvform}, as in the Einstein
  scale $\V_{ab}=0$ and hence the Cotton tensor vanishes
  everywhere. Thus in this scale $\Omega_{ab}{}^C{}_D V_C=0$, for all
  sections of the distinguished subbundle. But that equation is
  projectively invariant.

\end{itemize}

It follows from this that if $\na$ is some representative of $p$, and
\[
V_A=\horiztract{-\sigma\Upsilon_i}{\sigma},
\]
is a section of the rank-1 sub-bundle, then we have that
\[
0=\Omega_{ab}{}^C{}_DV_C=(W_{ab}{}^l{}_k\Upsilon_l+ C_{kab})\sigma.
\]
\end{proof}

\begin{proof}[Computational proof]
We show that for an arbitrary change of connection
\[
\nabla'_i\omega_j=\nabla_i\omega_j-\Upsilon_i\omega_j-\Upsilon_j\omega_i,
\]
the Cotton tensor changes by
\begin{equation}\label{Cotton-change}
2\nabla'_{[i}\V'_{j]k}=2\nabla_{[i}\V_{j]k}+W_{ij}{}^l{}_k\Upsilon_l.
\end{equation}
If the new connection $\nabla'$ is the Levi-Civita connection of some Einstein metric, the tensor $\V'$ is a constant multiple of that metric, hence parallel, and we then obtain
\[
0=2\nabla_{[i}\V_{j]k}+W_{ij}{}^l{}_k\Upsilon_l
\]
as required.

Indeed, recall from \nn{pcurvtrans} that under such a projective change of connection, the projective Weyl tensor $W_{ij}{}^k{}_l$ is invariant, and the projective Schouten tensor changes by
$$
\V'_{ij} = \V_{ij} - \nabla_i\Upsilon_j + \Upsilon_i\Upsilon_j.
$$
So by \nn{ptrans} its covariant derivative changes by
\begin{eqnarray*}
\nabla'_i\V'_{jk}
   &=& \nabla_i\V'_{jk}
      -2\V'_{jk}\Upsilon_i-\V'_{ji}\Upsilon_{k}-\V'_{ik}\Upsilon_{j}.
\end{eqnarray*}
Skewing over $i$ and $j$ gives
\[
\nabla'_{[i}\V'_{j]k}  = \nabla_{[i}\V'_{j]k} -\Upsilon_{[i}\V'_{j]k} +\tfrac{1}{2}\beta'_{ji}\Upsilon_k.
\]

We expand the three terms separately (using the transformation formula \nn{pcurvtrans} for the last one):
\begin{eqnarray*}
\nabla_{[i}\V'_{j]k}
     &=& \nabla_{[i}\V_{j]k}-\nabla_{[i}\nabla_{j]}\Upsilon_k+\Upsilon_k\nabla_{[i}\Upsilon_{j]}
    + \Upsilon_{[j}\nabla_{i]}\Upsilon_k \\
     &=& \nabla_{[i}\V_{j]k}+\tfrac{1}{2}R_{ij}{}^l{}_k\Upsilon_l+\Upsilon_k\nabla_{[i}\Upsilon_{j]}+\Upsilon_{[j}\nabla_{i]}\Upsilon_k \\
     &=& \nabla_{[i}\V_{j]k}+\tfrac{1}{2}W_{ij}{}^l{}_k\Upsilon_l
+\Upsilon_{[i}\V_{j]k}+\tfrac{1}{2}\beta_{ij}\Upsilon_k+\Upsilon_k\nabla_{[i}\Upsilon_{j]}+\Upsilon_{[j}\nabla_{i]}\Upsilon_k \\
-\Upsilon_{[i}\V'_{j]k}
     &=& -\Upsilon_{[i}\V_{j]k} +\Upsilon_{[i}\nabla_{j]}\Upsilon_k - \Upsilon_{[j}\Upsilon_{i]}\Upsilon_k  \\
     &=& -\Upsilon_{[i}\V_{j]k} +\Upsilon_{[i}\nabla_{j]}\Upsilon_k \\
 \tfrac{1}{2}\beta'_{ji}\Upsilon_k
&=&  \tfrac{1}{2}\beta_{ji}\Upsilon_k+\nabla_{[j}\Upsilon_{i]}\Upsilon_k.
\end{eqnarray*}
Summing, we obtain:
\[
2\nabla_{[i}\V'_{j]k}=2\nabla_{[i}\V_{j]k}+W_{ij}{}^l{}_k\Upsilon_l
\]
as required.

\end{proof}

\subsection{Obstructions in the generic setting}

By analogy with \cite{GN}, we call a torsion-free connection
\emph{weakly generic}, if the bundle map
\[
W :T^*M\to \Lambda^2T^*M\otimes T^*M
\]
defined by its projective Weyl tensor $W_{ij}{}^l{}_k$ has trivial kernel 
at every point.  We observe that this condition is projectively invariant.

If $\nabla$ is weakly generic, we may locally choose a left inverse
field $D^{ij}{}_m{}^k$ for $W_{ij}{}^l{}_k$, so that
\begin{equation}\label{Ddef}
D^{ij}{}_m{}^kW_{ij}{}^l{}_k=\delta_m{}^k.
\end{equation}  A weakly generic
connection has at most one solution $\Upsilon_i$ to the C-space
equation (and thus on weakly generic structures there is at most one
Einstein Levi-Civita connection in $p$). In terms of
$D^{ij}{}_m{}^k$ this solution, if it exists, is the covector field
\begin{equation}
\Upsilon_i=-D^{ab}{}_i{}^cC_{cab}.\label{upsilon-def}
\end{equation}
In the following we shall construct tensors using $D$ but which are
otherwise canonical to the projective structure $(M,p)$. We shall
refer to these as projective invariants even though by dint of the
choice $D$ they are not natural. In some cases there exist natural
$D$, whence in these cases we obtain natural projective
invariants. Funding such $D$ is taken up in Section \ref{natural}
below.

Now for each $D$ satisfying \nn{upsilon-def} we may construct tensors
that partly capture the projective-Einstein condition as
follows. First observe that we may  use \nn{upsilon-def} to replace
$\Upsilon$s in the expression
\[
\V_{ij} - \nabla_i\Upsilon_j + \Upsilon_i\Upsilon_j
\]
for $\V'_{ij}$ to define a 2-tensor
\begin{equation}\label{Gdef}
 G_{ij}:= \V_{ij} + \nabla_i(D^{ab}{}_j {}^cC_{cab})  +
D^{ab}{}_i{}^cC_{cab} D^{ef}{}_j{}^gC_{gef} .
\end{equation}

\begin{lemma}\label{G-inv}
For a fixed choice of $D$, the tensor $G_{ij}$ is projectively invariant.
\end{lemma}

\begin{proof}
Consider an arbitrary projective change of connection 
\[
\nabla'_i\omega_j=\nabla_i\omega_j-\Upsilon_i\omega_j-\Upsilon_j\omega_i.
\]
The expression (\ref{Gdef}) for the tensor $G'_{ij}$ has three terms; we calculate it term by term.

For the first term,
\[
\V'_{ij}=\V_{ij} - \nabla_i\Upsilon_j + \Upsilon_i\Upsilon_j.
\]

Before continuing, note that by equation (\ref{Cotton-change}), in the second proof of Proposition \ref{c-space}, we know that $C_{cab}$ changes by
$$
C'_{cab}=C_{cab}+W_{ab}{}^l{}_c\Upsilon_l.
$$
So $D^{ab}{}_i{}^cC_{cab}$ changes by
\begin{eqnarray*}
D^{ab}{}_i{}^cC'_{cab}&=&D^{ab}{}_i{}^cC_{cab}+D^{ab}{}_i{}^cW_{ab}{}^l{}_c\Upsilon_l\\
&=&D^{ab}{}_i{}^cC_{cab}+\Upsilon_i.
\end{eqnarray*}

Now, for the second term,
\begin{eqnarray*}
\nabla'_i(D^{ab}{}_j {}^cC'_{cab})&=&
\nabla_i(D^{ab}{}_j {}^cC'_{cab})-\Upsilon_iD^{ab}{}_j {}^cC'_{cab}-D^{ab}{}_i {}^cC'_{cab}\Upsilon_j\\
&=&\nabla_i(D^{ab}{}_j {}^cC_{cab})+\nabla_i\Upsilon_j-\Upsilon_iD^{ab}{}_j {}^cC_{cab}-D^{ab}{}_i {}^cC_{cab}\Upsilon_j-2\Upsilon_i\Upsilon_j.
\end{eqnarray*}

For the third term,
$$
D^{ab}{}_i{}^cC'_{cab} D^{ef}{}_j{}^gC'_{gef} =
D^{ab}{}_i{}^cC_{cab} D^{ef}{}_j{}^gC_{gef}
+D^{ab}{}_i{}^cC_{cab} \Upsilon_j
+\Upsilon_i D^{ef}{}_j{}^gC_{gef}
+\Upsilon_i\Upsilon_j.
$$

Summing, we find that
$$
G'_{ij}= \V_{ij} + \nabla_i(D^{ab}{}_j {}^cC_{cab})  +
D^{ab}{}_i{}^cC_{cab} D^{ef}{}_j{}^gC_{gef} =G_{ij},
$$
as required.
\end{proof}

  Now if $p$ contains a Levi-Civita connection $\nabla'$ for some Einstein
  metric $g$ then \nn{upsilon-def} finds the 1-form which by
  \nn{ptrans} gives the projective transformation to that connection.
Thus
by construction in that case $G_{ij}=\lambda g_{ij}$ for
 some constant $\lambda\in \mathbb{R}$. On the other hand starting
 with a weakly generic $(M,p)$ the properties of the invariant
 $G_{ij}$ may forbid the existence of Einstein Levi-Civita connection
 in $p$: if the skew part $G_{[ij]}$ is non-zero then $(M,p)$ is not
 projectively Einstein; if $G_{ij}$ is non-zero but degenerate then
 $(M,p)$ is not projectively Einstein;  $G_{ij}$ is zero if and only if
 $(M,p)$ is projectively Ricci-flat (but not necessarily projectively
 metric). In relation to the second of these points let us define
the density $\gamma\in \ce(-2(n+1))$ by 
$$
G_{a_1b_1}G_{a_2b_2}\cdots G_{a_nb_n}
$$ 
(and the usual identification of $(\Lambda^nT^*M)^2$ with $\ce(-2(n+1))$)
where the sequentially labelled indices are skewed over.

To construct the next obstruction we again consider $\nabla'\in p$.
We want to test whether this is an Einstein Levi-Civita connection. If
so then by the argument above then $G_{ij}$ is a constant times the
Einstein metric whence $\nabla'_i G_{jk}$ would be zero. We calculate
at some $\nabla\in p$. The first step is that we expand the formula
which expresses $\nabla'_iG_{jk} $ in terms of $\nabla$, and the
1-form $\Upsilon_i$ that relates $\nabla$ and $\nabla'$ as in
\nn{ptrans}. Using \nn{ptrans} we have
$$
\nabla'_i G_{jk}=
\nabla_iG_{jk}
      -2G_{jk}\Upsilon_i-G_{ji}\Upsilon_{k}-G_{ik}\Upsilon_{j}
$$
and in here we use \nn{upsilon-def} to replace the $\Upsilon$s.
This yields a projective invariant
\begin{equation}\label{Edef}
E_{ijk}:=\nabla_iG_{jk}
+(2G_{jk}D^{ab}{}_i{}^c
+G_{ji}D^{ab}{}_k{}^c
+G_{ik}D^{ab}{}_j{}^c)C_{cab},
\end{equation}
which obstructs the existence of an Einstein connection projectively
equivalent to $\nabla$. This polynomially involves $D$, $W$ and $C$ and their
$\nabla$-covariant derivatives to second order.
\begin{remark}\label{pcot}
Since $E_{ijk}$ is projectively invariant so is $2E_{[ij]k}$. This has
a nice interpretation: it is exactly the obstruction to the existence
of a Cotton flat affine connection in the projective class. Indeed
there is a simple formula for this
$$
2E_{[ij]k}= C_{kij}- W_{ij}{}^\ell{}_k D^{ab}{}_\ell{}^cC_{cab},
$$ 
obtained by inserting \nn{upsilon-def} into \nn{Cotton-change}.
\end{remark}

\begin{lemma}
For a fixed choice of $D$, the tensor $E_{ijk}$ is projectively invariant.
\end{lemma}

\begin{proof}
Consider an arbitrary projective change of connection
\[
\nabla'_i\omega_j=\nabla_i\omega_j-\Upsilon_i\omega_j-\Upsilon_j\omega_i.
\]
The expression (\ref{Edef}) for the tensor $E'_{ijk}$ has four terms; we calculate it term by term.  This is simplified by the projective invariance of $G_{ij}$ (Lemma \ref{G-inv}).

For the first term,
\[
\nabla'_iG_{ij}=\nabla_iG_{jk}
      -2G_{jk}\Upsilon_i-G_{ji}\Upsilon_{k}-G_{ik}\Upsilon_{j}.
\]

Before continuing, note as in the proof of Lemma \ref{G-inv} that by equation (\ref{Cotton-change}), the quantity $D^{ab}{}_i{}^cC_{cab}$ transforms by
$$
D^{ab}{}_i{}^cC'_{cab}=D^{ab}{}_i{}^cC_{cab}+\Upsilon_i.
$$

Now, for the second, third and fourth terms,
\begin{eqnarray*}
2G_{jk}D^{ab}{}_i{}^cC'_{cab}
&=&2G_{jk}D^{ab}{}_i{}^cC_{cab} +2G_{jk}\Upsilon_i,\\
G_{ji}D^{ab}{}_k{}^cC'_{cab}
&=&G_{ji}D^{ab}{}_k{}^cC_{cab} +G_{ji}\Upsilon_k,\\
G_{ki}D^{ab}{}_j{}^cC'_{cab}
&=&G_{ik}D^{ab}{}_j{}^cC_{cab} +G_{ik}\Upsilon_j.
\end{eqnarray*}

Summing, we find that
$$
E'_{ijk}=\nabla_iG_{jk}
+(2G_{jk}D^{ab}{}_i{}^c
+G_{ji}D^{ab}{}_k{}^c
+G_{ik}D^{ab}{}_j{}^c)C_{cab}
=E_{ijk},
$$
as required.
\end{proof}

\begin{proposition} \label{Dmain}
Let $\nabla$ be a weakly generic torsion-free connection, and let $D$
be a left inverse for $W$ in the sense of \nn{Ddef}. Then $\nabla$ is
projectively equivalent to a Ricci-flat affine connection if and only if
$G_{ij}$ is zero. Moreover the following are equivalent:
\begin{enumerate}
\item $\nabla$ is projectively equivalent to the Levi-Civita connection of an Einstein metric with nonvanishing Einstein constant.
\item $E_{ijk}$ and $G_{[ij]}$ are each zero while $\gamma$ is nowhere
  zero (i.e. $G_{ij}$ is everywhere nondegenerate).
\end{enumerate}
\end{proposition}

\begin{proof} The first claim was discussed earlier.
The discussion before the Proposition shows that (1) implies (2).  For
the converse, given a connection for which $E_{ijk}$ vanishes and
$G_{ij}$ is symmetric and nondegenerate, we \emph{define} a tensor
$\Upsilon_i$ by (\ref{upsilon-def}).
But if  $G_{ij}$ is symmetric and nondegenerate then it is a metric of some
signature. 
 We
define $\nabla'$ to be the connection obtained from $\nabla$ and the
projective change (\ref{ptrans}) using $\Upsilon_i$.

The vanishing of $E_{ijk}$ implies that, in the projective
change of connection defined by $\Upsilon_i$,
\[
\nabla'_iG_{jk}=E_{ijk}=0.
\]
So $\nabla'$ is the Levi-Civita connection of $G_{ij}$. On the other
hand by construction $G_{ij}$ is equal to $\V'_{ij}$, the Schouten tensor for $\nabla'$.
Thus
\[
\text{Ric}(G_{ij})=(n-1)\V'_{ij}=(n-1)G_{ij},
\]
so $G_{ij}$ and all of its non-zero multiples are Einstein, and their common Levi-Civita connection $\na'$ is
projectively equivalent to $\nabla$.
\end{proof}

Again, for completeness, we check that these obstructions are nontrivial.
\begin{proposition}
There exist torsion-free weakly generic connections.
\end{proposition}
\begin{proof}
We prove this in dimension 4.  Let $\nabla$ be the Levi-Civita
connection of a metric which is Ricci-flat but not locally conformally
flat.  The conformal Weyl curvature $\widetilde{W}$ satisfies,
$$
4 \widetilde{W}_{ij}{}^k{}_l\widetilde{W}^{ij}{}_k{}^m =|\widetilde{W}|^2\delta^m_l,
$$
so the map
$$
\widetilde{W} :T^*M\to \Lambda^2T^*M\otimes T^*M
$$
is invertible.  Since $\nabla$ is Ricci-flat, the two Weyl curvatures agree:  $W=\widetilde{W}$; see Proposition \ref{proj-conf-Weyl} below.
\end{proof}

\subsection{Natural left inverses} \label{natural}

Proposition \ref{Dmain} reduces the problem of finding a sharp
obstruction to the projective Einstein problem to that of finding
natural left inverses for $W_{ab}{}^c{}_d$, the tensors $D$ in the
Proposition. More precisely, this is true in the case of nonzero
Einstein constant.  In this section we construct such natural
tensors.

{F}irst, suppose the dimension of the manifold is even, say $n=2m$.  Write
$(Q^k)_{a_1a_2\cdots a_{2k}}{}^i{}_j$
for the totally alternating part of
\[
W_{a_1a_2}{}^{c_1}{}_{j}W_{a_3a_4}{}^{c_2}{}_{c_1}\cdots W_{a_{2k-1}a_{2k}}{}^{i}{}_{c_{k-1}}
\]
(that is, the result of skewing over all $a_i$-indices), so that its trace $(Q^k)_{a_1a_2\cdots a_{2k}}{}^i{}_i$ is the $k$-th curvature form.
Each tensor
\[
(Q^m)_{a_1a_2\cdots a_{2k}}{}^l{}_r
\]
is a section of $\Lambda^{2k}T^*M\otimes \text{End}(TM)$. For the two
highest $k$ among such tensors we have the following: Via the natural
(``Hodge-star'') isomorphism
\[
\Lambda^k(T^*M)\to \Lambda^{2m-k}(TM)\otimes\Lambda^{2m}(T^*M),
\]
we may in fact treat $Q^{m-1}$ as a section $(Q^{m-1})^{b_1b_2}{}_r{}^l$ of the 
 bundle
\[
\Lambda^2(TM)\otimes \End(TM)\otimes \Lambda^{2m}(T^*M)
\]
and $Q^m$ as a section $(Q^m)_r{}^l$ of the bundle
\[
\End(TM)\otimes \Lambda^{2m}(T^*M).
\]
We note that
\[
(Q^{m-1})^{b_1b_2}{}_r{}^sW_{b_1b_2}{}^l{}_s=(Q^m)_r{}^l.
\]
Since $\End(TM){}\otimes \Lambda^{2m}(T^*M)$ is just a twisting of the
endomorphism bundle by the line bundle $\Lambda^{2m}(T^*M)$, we may
meaningfully speak
\begin{itemize}
\item of $Q^m$'s determinant $||Q^m||$, which is a section of $(\Lambda^{2m}(T^*M))^{2m}$;
\item of $Q^m$'s pointwise adjugate $\widetilde{(Q^m)}{}_r{}^l$, which is a section of
\[
\End(TM){}\otimes (\Lambda^{2m}(T^*M))^{2m-1};
\]
\item of $Q^m$'s invertibility, which occurs precisely where
  $||Q^m||\neq 0$; the inverse is then the section
  $||Q^m||^{-1}\widetilde{(Q^m)}{}_r{}^l$ of $\End(TM)\otimes
  (\Lambda^{2m}(TM))$.
\end{itemize}

For connections $\nabla$ satisfying the further ``genericity'' condition that $||Q^m||$ does not vanish, we thus obtain a natural left inverse $D$ for the Weyl curvature $W$: the tensor
\[
D_{(Q)}^{b_1b_2}{}_i{}^k:=||Q^m||^{-1}\widetilde{(Q^m)}_i{}^r(Q^{m-1})^{b_1b_2}{}_r{}^k.
\]
Indeed,
\begin{eqnarray*}
D_{(Q)}^{b_1b_2}{}_i{}^kW_{b_1b_2}{}^j{}_k
&=& ||Q^m||^{-1}\widetilde{(Q^m)}_i{}^r(Q^{m-1})^{b_1b_2}{}_r{}^kW_{b_1b_2}{}^j{}_k\\
&=& ||Q^m||^{-1}\widetilde{(Q^m)}_i{}^r(Q^m)_r{}^j\\
&=& \delta_i{}^j.
\end{eqnarray*}

This is just one example of a large family of natural left inverses
for $W_{ab}{}^c{}_d$.  A general procedure which works in either
dimension parity is as follows.  First construct a family of tensors
$(Q^{\underline{N}, F}) ^{b_{1} \cdots b_R }{}_t{}^s$: the inputs are
\begin{enumerate}
\item A nonnegative integer $R$, with $R<n$, and an even natural number $2N$, such that $2N+R$ is a multiple of $n$, the dimension of the manifold;
\item A partition $N_0+\cdots+N_r=N$ of $N$;
\item 
A function \[
F:\coprod_{j=0}^r\{1,\ldots, 2N_j\}\to
\{1,\ldots, (2N+R)/n\}
\]
 with the property that for each $i\in\{2,\ldots, (2N+R)/n\}$ (but not for $i=1$), $|F^{-1}(i)|=n$.  Thus $|F^{-1}(1)|=n-R$.
\end{enumerate}
We introduce the notation
\[
(P^k)_{a_1\cdots a_{2k}}{}^s{}_t:=W_{a_1a_2}{}^{c_1}{}_{t}W_{a_3a_4}{}^{c_2}{}_{c_1}\cdots W_{a_{2k-1}a_{2k}}{}^s{}_{c_{k-1}}.
\]
(Thus the full skew of such a tensor's trace, $(P^k)_{[a_1\cdots a_{2k}]}{}^s{}_s$, is the $k$-th curvature form $(p_k)_{a_1\cdots a_{2k}}$.)
Consider the tensor
\[
(P^{N_0})_{a^0_1\cdots a^0_{2N_0}}{}^s{}_t
\prod_{j=1}^r (P^{N_j})_{a^j_1\cdots a^j_{2N_j}}{}^{c_j}{}_{c_j}.
\]
For each $i\in\{2,\ldots, (2N+R)/n\}$, skew the $n$ indices $\{a^j_\alpha:F(j,\alpha)=i\}$ of this tensor, and also (for $i=1$) skew the $n-R$ indices $\{a^j_\alpha:F(j,\alpha)=1\}$.  Finally, again via the natural
(``Hodge-star'') isomorphism, we may identify the result with a section of
\[
\Lambda^R(TM)\otimes \End(TM)\otimes (\Lambda^n(T^*M))^{(2N+R)/n}.
\]
This section is the tensor $(Q^{\underline{N}, F})
^{b_{1} \cdots b_R }{}_t{}^s$.

To construct obstructions, now take a valid set of inputs $(N, 0, \underline{N}, F)$ as above, i.e.\ a valid set of inputs in which $R=0$.  This means:
\begin{enumerate}
\item $2N$ is an even number which is a multiple of $n$, the dimension of the manifold.
\item $N_0+\cdots+N_r=N$ is a partition of $N$.
\item 
\[
F:\coprod_{j=0}^r\{1,\ldots, 2N_j\}\to\{1,\ldots, 2N/n\}
\] 
is a function with, for each $i\in\{1,\ldots, 2N/n\}$, $|F^{-1}(i)|=n$.
\end{enumerate}
(The inputs $N=m$, $\underline{N}=(m)$, $F\equiv 1$ will yield the
special case considered at the start of this subsection.)  Write
\[
\underline{N}'=(N_0-1,N_1,\cdots N_r),
\quad F'=F|_{\{1,\ldots, 2N_0-2\}\amalg\left(\coprod_{j=1}^r\{1,\ldots, 2N_r\}\right)}.
\]
With $N'=N-1$ and $R'=2$, the set of inputs $(N', 2, \underline{N}', F')$ is then also valid.

We note that
\[
(Q^{\underline{N}', F'})
^{b_{1}b_2}{}_t{}^l
W_{b_1b_2}{}^s{}_l
=
(Q^{\underline{N}, F}){}_t{}^s.
\]
Since $\End(TM)\otimes (\Lambda^n(T^*M))^{2N/n}$
is just a twisting of the
endomorphism bundle by the line bundle $(\Lambda^{n}(T^*M))^{2N/n}$, we may
meaningfully speak
\begin{itemize}
\item of $Q^{\underline{N}, F}$'s determinant $||Q^{\underline{N}, F}||$, which is a section of $(\Lambda^{n}(T^*M))^{2N}$;
\item of $Q^{\underline{N}, F}$'s pointwise adjugate $\widetilde{(Q^{\underline{N}, F})}{}_r{}^l$, which is a section of
\[
\End(TM)\otimes (\Lambda^n(T^*M))^{2N(n-1)/n}
\]
\item of $Q^{\underline{N}, F}$'s invertibility, which occurs precisely where
  $||Q^{\underline{N}, F}||\neq 0$; the inverse is then the section
  $||Q^{\underline{N}, F}||^{-1}\widetilde{(Q^{\underline{N}, F})}{}_r{}^l$ of 
$\End(TM)\otimes (\Lambda^n(TM))^{2N/n}$.
\end{itemize}

For connections $\nabla$ satisfying the further ``genericity'' condition that $||Q^{\underline{N}, F}||$ does not vanish, we thus obtain a natural left inverse $D$ for the Weyl curvature $W$: the tensor
\[
D_{(Q)}^{b_1b_2}{}_i{}^k:=||Q^{\underline{N}, F}||^{-1}\widetilde{(Q^{\underline{N}, F})}_i{}^r(Q^{\underline{N}', F'})^{b_1b_2}{}_r{}^k.
\]
Indeed,
\begin{eqnarray*}
D_{(Q)}^{b_1b_2}{}_i{}^kW_{b_1b_2}{}^j{}_k
&=& ||Q^{\underline{N}, F}||^{-1}\widetilde{(Q^{\underline{N}, F})}_i{}^r(Q^{\underline{N}', F'})^{b_1b_2}{}_r{}^kW_{b_1b_2}{}^j{}_k\\
&=& ||Q^{\underline{N}, F}||^{-1}\widetilde{(Q^{\underline{N}, F})}_i{}^r(Q^{\underline{N}, F})_r{}^j\\
&=& \delta_i{}^j.
\end{eqnarray*}

In summary, adapting Proposition \ref{Dmain} we have the following.
\begin{theorem} \label{appmain2}
For a $||Q^{\underline{N}, F}||$-nowhere-zero torsion-free connection $\nabla$, the
natural projective invariant $G^{(Q)}_{ij}$ completely obstructs the
projective class containing a Ricci-flat connection.  Moreover the
following are equivalent:
\begin{enumerate}
\item $\nabla$ is projectively equivalent to the Levi-Civita
  connection of an Einstein metric with nonvanishing Einstein
  constant.
\item The natural projective invariants $E^{(Q)}_{ijk}$ and
  $G^{(Q)}_{[ij]}$ are each zero while $\gamma^{(Q)}$ is nowhere zero
  (i.e. $G^{(Q)}_{ij}$ is everywhere nondegenerate).
\end{enumerate}
\end{theorem}

\begin{remark}\label{genericr}
It would be interesting to determine which valid input sets
$(N,0,\underline{N}, F)$ yield nontrivial obstructions, i.e.\ have the
property that $||Q^{\underline{N}, F}||$ is generically nonvanishing.
We expect that in both dimension parities this happens frequently. The
non-vanishing of any such $||Q^{\underline{N}, F}||$ defines a notion of {\em generic} Weyl curvature.
\end{remark}

\begin{remark}
By multiplying through by a suitable power of $||Q^{\underline{N},
  F}||$, these invariants can be made polynomial rather than rational
in the jets of the underlying affine connection.
\end{remark}

\section{Obstructions proliferating} \label{prol}

We now note that using the results established above there are many
ways available for producing projective obstructions to Einstein
Levi-Civita connections. In fact there is a general principle that is
very effective. We describe this here.  The principle exploits linear
identities, whereas the theory above provides a systematic approach to
producing such identities.  The idea behind this principle is
well-known.  For instance, in \cite[Corollary 3.5]{GN} it is used to
obtain obstructions to a conformal class being conformally Einstein,
and in \cite[Section 8]{BDE} (mentioned again in \cite[Theorem
  2.13]{NurMet}) it is used to obtain an obstruction to a projective
class being metric. This idea and some strategies from \cite{GN} are also
used effectively by Case in \cite{Ca}.

We formalise the principle in the following lemma.
\begin{lemma} \label{princ}
Let $(M,p)$ be a projective manifold, let $E$ and $F$ be natural
vector bundles on $(M,p)$, and let the field $A$ be a natural section
of $E\otimes F$.  Let $k$ be the rank of $E$.

Then the corresponding section $A^{\wedge k}$ of $\Lambda^k(E)\otimes
\Lambda^k(F)$, sharply obstructs the existence of nonvanishing local
sections $\xi$ of $E^*$ such that
$$
\langle A, \xi\rangle = 0.
$$
\end{lemma}
\begin{proof}
The section $A$ of $E\otimes F$ induces a natural bundle map $A:E^*\to F$.  Its $k$-th exterior power
$$
A^{\wedge k}:\Lambda^k(E^*)\to \Lambda^k(F),
$$
the ``top-dimensional minors'' bundle map, which may be identified with a section of
$$
\Lambda^k(E)\otimes \Lambda^k(F),
$$
vanishes precisely if $A:E^*\to F$ has nontrivial kernel.
\end{proof}

\begin{remark}
The Lemma is useful when $k\leq \rank(F)$. Otherwise if $\rank(E)>
\rank(F)$ then the bundle $\Lambda^k(E)\otimes \Lambda^k(F)$ has rank zero.
\end{remark}

We have already obtained several linear relationships to which this
lemma may be applied.  We recall them here.  Let $p$ be a projective
equivalence class which contains the Levi-Civita $\na$ connection of
an Einstein metric $g_{ab}$. Then we have the following:
\begin{enumerate}
\item It will be proved later (Proposition \ref{proj-conf-Weyl}) that the projective Weyl curvature $W$ of the class $p$ must agree with the conformal Weyl curvature $\widetilde{W}$ of $\nabla$.  This latter has the symmetries of $g_{ab}$'s Riemann curvature tensor; therefore
$$W_{ab}{}^c{}_{(d} g_{e)c} = 0.$$
\item Recall from the ``Tractor Proof'' of Proposition \ref{c-space} the existence of a distinguished nonvanishing section $V_A$ of $\cT^*$
such that
$$
0=\Omega_{ab}{}^C{}_DV_C.
$$

\item Recall from Section \ref{char} the existence of a section $h$ of $\text{Sym}^2(\cT)$, given in the scale connection $\na$ and the trivialisation of density bundles $\text{dVol}_g$ by
$$
h^{AB}=
\begin{pmatrix}
g^{ab} & 0 \\
0 & \frac{1}{n}\V
\end{pmatrix},
$$
which is parallel with respect to the tractor connection:
$$
\na h^{AB}=0.
$$
Differentiating again and skewing yields a linear relationship involving tractor curvature:
$$
\Omega_{ab}{}^{(C}{}_D h^{E)D}.
$$
\item Since the tractor $h^{AB}$ is parallel, taking jets of the
  linear relationship just obtained yields arbitrarily many further linear
  relationships.
\end{enumerate}
In each case we have a nonvanishing field (e.g.\ the metric in (1), $V$
in (2)) which satisfies a  homogeneous linear equation whose
coefficients are natural projectively invariant fields.

The point is that from each of these geometrically obtained linear
identities we obtain obstructions to the projective-Einstein problem.
  For instance, from the first relationship we
obtain a projectively invariant $W^{\wedge\frac{n(n+1)}{2}}$, a
section of the bundle
\begin{eqnarray*}
&& \Lambda^{\frac{n(n+1)}{2}}\left(\ce^{(ab)}\right)
\otimes\Lambda^{\frac{n(n+1)}{2}}\left(\ce_{[ab]}\otimes\ce_{(de)}\right) ,
\end{eqnarray*}
which obstructs the existence of an Einstein 
connection in the projective equivalence class.

From the other relationships, the projective invariants we obtain are
initially sections of mixed tensor-tractor bundles, but these may be
expanded into collections of (individually non-invariant) tensor
obstructions if desired.

\section{Conformal differential geometry} \label{confg}

A {\em conformal structure} (of signature $(p,q)$) $(M^n,c)$, $n\geq 3$, is a smooth manifold
equipped with an equivalence class $c$ of signature $(p,q)$-metrics,
where two metrics $g$ and $ g'$ in $c$ are equivalent if there
is some positive smooth function $\Omega$ such that
$g'=\Omega^2 g$.  That is, an equivalence class is a maximal
set of metrics which are mutually pointwise homothetic on each tangent
space.

Let $g$ and $ g'=\Omega^2 g$ be conformally related metrics,
and define a 1-form $\Upsilon_a:=\Omega^{-1}\nabla_a\Omega$ from their
conformal factor.  The two metrics determine Levi-Civita connections
$\nabla$, $\nabla'$, which are related by, for $ u_b\in
\ce_b$,
\begin{equation} \label{cconntrans}
\nd'_a u_b=\nd_a u_b -\Upsilon_a u_b-\Upsilon_b u_a +g_{ab} \Upsilon^c u_c 
\end{equation}

\subsubsection{Curvature tensors arising in conformal geometry}\label{cwstC}

Given a metric $g\in c$, letting $\nabla$ be its Levi-Civita connection, the (Riemannian) curvature is defined as
usual by 
$$
(\nd_a\nd_b-\nd_b\nd_a)v^c=R_{ab}{}^c{}_dv^d .
$$
This can be decomposed into the totally trace-free {\em conformal Weyl tensor} $\widetilde{W}_{ab}{}^c{}_d $ and a remaining part described by the symmetric {\em conformal Schouten tensor} $\widetilde{\V}_{ab}$, according to
\begin{equation} \label{cweyl}
R_{abcd}=\widetilde{W}_{abcd}+2g_{c[a}\widetilde{\V}_{b]d}+2g_{d[b}\widetilde{\V}_{a]c},
\end{equation}
where $[\cdots]$ indicates the antisymmetrisation over the enclosed indices.  We write $\widetilde\J$ for the trace $g^{ab}\widetilde\V_{ab}$.
The Schouten tensor is a trace modification of the Ricci tensor $\Ric_{bd}=R_{ab}{}^a{}_d$:
\begin{equation}\label{cschouten}
\Ric_{ab}=(n-2)\widetilde\V_{ab}+\widetilde\J g_{ab}.
\end{equation}

Under a conformal change of metric, one computes that the
Weyl curvature $\widetilde{W}_{ab}{}^c{}_d$ is unchanged. Thus it is an invariant
of the conformal structure $(M,c)$. (In dimension 3 this vanishes.)

\subsection{Conformal densities and connections thereon}\label{cdense}
For our subsequent discussion it is convenient to take the positive
$(2n)^{th}$ root of $(\Lambda^nTM)^2$ and we denote this
$\ce[1]$. Then for $w\in \mathbb{R}$ we denote by $\ce [w]$ its
$w^{\rm th}$-power. Sections of $\ce [w]$ will be described as {\em
  conformal densities} of weight $w$.  Given any bundle $\cB$
we shall write $\mathcal{B}[w]$
as a shorthand notation for $\mathcal{B}\otimes \ce[w]$.

Now we consider a conformal manifold $(M,c)$.  Each metric
$g\in c$ determines a metric (also denoted $g$) on
$(\Lambda^nTM)^2$ and hence on its roots $\ce[w]$, $w\in \mathbb{R}$.  Moreover the Levi-Civita connection $\nabla$ of $g$ determines compatible, flat connections (also denoted $\nabla$) on $(\Lambda^nTM)^2$ and on its roots $\ce[w]$, $w\in \mathbb{R}$.

The conformal class $c$ determines canonical sections $\bg_{ab}$ of
$\ce_{(ab)}[2]$ and $\bg^{ab}$ of $\ce^{(ab)}[-2]$.

\subsection{Conformal tractor calculus}\label{cT}
In analogy with the case of projective manifolds, on a conformal
manifolds there is a canonical tractor bundle equipped with metric and
connection \cite{BEG}, with historical precedents as in the projective
case. This is a rank-$(n+2)$ bundle that is closely related to $TM$.

To construct this bundle (in fact we construct the dual of the bundle usually considered),  we consider the jet exact sequence at
2-jets of the density bundle $\ce[1]$:
$$
0\to \ce_{(ab)}[1]\to J^2(\ce[1])\to J^1(\ce[1])\to 0,
$$
where $ (\cdots)$ indicates symmetrisation over the enclosed
indices.  Note we have a bundle homomorphism $ \ce_{(ab)}[1] \to \ce[-1]$
given by complete contraction with the conformal class $\bg^{ab}$. This is split via 
$ \rho\mapsto  \frac{1}{n}\rho \bg_{ab}$ and so 
the conformal structure decomposes $\ce_{(ab)}[1]$ into the
direct sum $\ce_{(ab)_0}[1]\oplus\ce[-1]$.  
Clearly then the $c$-tracefree bundle $\ce_{(ab)_0}[1]$ is a smooth subbundle of $J^2(\ce[1])$, and we
define $\ce_\alpha$ to be the quotient bundle.  That is, the {\em conformal
  cotractor bundle} $\widetilde{\ct}^*$ or $\ce_\alpha$ is defined by the exact sequence
\begin{equation}\label{ctrdef}
0\to \ce_{(ab)_0}[1]\to J^2(\ce[1])\to \ce_\alpha\to 0.
\end{equation}
The jet exact sequence at 2-jets, and the corresponding sequence at
1-jets, viz $ 0\to \ce_{a}[1]\to J^1(\ce[1])\to \ce[1]\to 0 , $
determine a composition series for $\ce_\alpha$ which we can summarise 
via a self-explanatory semi-direct sum notation $ \ce_\alpha= \ce[-1]\rpl
\ce_a[1]\rpl\ce[1]$.

A choice of metric $g$ from the conformal class determines canonical
splittings of this exact sequence, and hence a canonical
identification of $\ce_\alpha$ with the direct sum $\ce[-1]\oplus
\ce_a[1]\oplus\ce[1]$.

The conformal cotractor bundle has an invariant metric
$\widetilde{h}^{\alpha \beta}$ of signature $(p+1,q+1)$, the {\em
  (conformal) tractor metric}, and an invariant connection
$\widetilde\nabla_a$ preserving $\widetilde{h}^{\alpha \beta}$, the
{\em conformal tractor connection}. If for a metric $g$ from the
conformal class $V_\alpha, V_\beta \in\ce_\alpha$ are given by
\[
V_\alpha\stackrel{g}{=}(\tau  ~\mid ~\mu_a  ~\mid ~\si),
\quad
\underline{V}_\beta
\stackrel{g}{=}(\underline{\tau}  ~\mid ~\underline{\mu}_b
 ~\mid ~ \underline{\sigma}),
\]
 then the tractor metric is given by
\begin{equation}\label{cstdtracmetr}
\widetilde{h}^{\alpha\beta}V_\alpha\underline{V}_\beta=\bg^{ab}\mu_a\underline{\mu}_b+\sigma\underline{\tau}+
\tau\underline{\sigma},
\end{equation}
and the tractor connection is given by
\renewcommand{\arraystretch}{1}
\begin{equation}\label{cstdtracconn}
\widetilde\nabla_a V_\beta\stackrel{g}{=}
\left(\begin{array}{c} \nabla_a \tau - \widetilde\V_{ab}\bg^{bc}\mu_c \\
                       \nabla_a \mu_b+ \bg_{ab} \tau +\widetilde\V_{ab}\si \\
\nabla_a \si-\mu_a 
                        \end{array}\right)^T . 
\end{equation}
\renewcommand{\arraystretch}{1.5}

\section{The projective-conformal connection}\label{pun}

Given a metric $g$, with Levi-Civita connection $\nabla$, one may
compare the conformal geometry of the class $[g]$ and the projective
geometry of the class $[\nabla]$.

The main result of this section is that, when $g$ is Einstein, the
connection  between these two geometries is
particularly simple.  These ideas motivated our original
(re-)construction of Armstrong's projective tractor sub-metrics,
Theorem \ref{charth}.

For use in this section, we note the relationships between
corresponding conformal and projective curvature tensors of a metric
$g_{ij}$ which is Einstein.  Let $\lambda$ be the constant such that
$\Ric_{ij}=\lambda g_{ij}$.

\begin{lemma}\label{proj-conf-Schouten}
The conformal Schouten tensor of $g$ and projective Schouten tensor of
$\nabla$ simplify to, respectively, $\widetilde{\V}_{ij} =
\tfrac{1}{2(n-1)}\lambda g_{ij}$ and $\V_{ij} = \tfrac{1}{n-1}\lambda
g_{ij}$.  In particular, $\widetilde{\V}_{ij}=\tfrac{1}{2}\V_{ij}$.
\end{lemma}

\begin{proof} This follows from the definitions (\ref{cschouten}) and (\ref{pschouten}).
\end{proof}

For the following see e.g.\ \cite[Corollary 2.6]{NurMet}.
\begin{proposition}
\label{proj-conf-Weyl}
The conformal Weyl curvature of $g$ is the same as the projective Weyl
curvature of its Levi-Civita connection $\nabla$.
\end{proposition}

\begin{proof}
By Lemma \ref{proj-conf-Schouten}, and the definitions (\ref{cweyl})
and (\ref{pweyl}),
\begin{eqnarray*}
\widetilde{W}_{ijkl}&=&R_{ijkl}-2g_{k[i}\widetilde{\V}_{j]l}-2g_{l[j}\widetilde{\V}_{i]k}
\\ &=&R_{ijkl}-\tfrac{2}{n-1}\lambda g_{k[i}g_{j]l};
\\ W_{ij}{}^k{}_l&=&R_{ij}{}^k{}_l -
2\delta_{[i}{}^k\V_{j]l}\\ &=&R_{ij}{}^k{}_l - \tfrac{2}{n-1}\lambda
\delta_{[i}{}^kg_{j]l}.
\end{eqnarray*}
These agree up to the raising of an index of $\widetilde{W}$.
\end{proof}

\begin{corollary}\label{n3ob}
On projective 3-manifolds $(M,p)$ the projective Weyl curvature
sharply obstructs the existence of a Einstein Levi-Civita connection
in the projective class.
\end{corollary}
\begin{proof}
In dimension 3 the conformal Weyl tensor is identically zero.
\end{proof}

\begin{remark}
One can also see this from the fact that in dimension 3, a metric is
Einstein if and only if it has constant sectional curvature. Hence any
connection projectively related to the Levi-Civita connection of an
Einstein metric is also projectively related to a flat connection.
\end{remark}

\subsection{Conformal and projective Einstein models}\label{model}

By Proposition \ref{proj-conf-Weyl}, an Einstein metric is conformally
flat precisely if its Levi-Civita connection is projectively flat.
Such a metric has constant sectional curvature.

The model conformally flat manifold of dimension $n$ and its tractor
bundle are constructed as follows \cite{Gal}.  Let $\widetilde{V}$ be
a vector space of dimension $(n+2)$, equipped with a nondegenerate
symmetric bilinear form $\langle\cdot, \cdot\rangle$ of indefinite
signature $(p+1,q+1)$.  Let
\[
C = \{w\in \widetilde{V} : \langle w, w\rangle=0\}
\]
be the null cone of $\widetilde{V}$.  The ray projectivisation
$\mathbb{P}C$ is diffeomorphic to $S^p\times S^q$ and the bilinear
form $\langle\cdot, \cdot\rangle$ descends to a well-defined, flat
conformal equivalence class $c$ of metrics on $\mathbb{P}C$.  For any
open, ray-closed subset $U\subseteq C$, choices of metric $g\in
c|_{\mathbb{P}U}$ are in bijection with sections of the
$(\mathbb{R}^+)$-principal bundle $U\to\mathbb{P}U$.

The conformal tractor bundle may be identified with the space
$\widetilde{V}$, and parallel tractors with elements of
$\widetilde{V}$.

Einstein (pseudo-)metrics on regions of $\mathbb{P}C$ are in bijection
with choices of nonvanishing vector $I\in \widetilde{V}$.  Indeed,
such a vector determines a section of some part of
$C\to\mathbb{P}C$, via the intersection of the hyperplane $\langle I,
\cdot\rangle=1$ with the null cone.

Einstein metrics of positive, zero and negative curvature correspond
respectively to timelike, null and spacelike vectors $I$.

The model projectively flat manifold of dimension $n$ is the ray
projectivisation of a dimension-$(n+1)$ vector space $V$.  The total
space of the projective tractor bundle may be identified with the
space $V$.

Thus, for a choice $I$ of nonvanishing vector, determining a choice of
Einstein metric in the model conformal geometry, the projective
tractor geometry of this Einstein metric is naturally embedded in the
conformal tractor bundle $\widetilde{V}$ as the hyperplane $I^\perp$.

\subsection{A relationship of tractor bundles}

Let $g_{ab}$ be an Einstein metric, with Levi-Civita connection
$\nabla$.  In this section we compare the conformal tractor geometry
of $[g_{ab}]$ and the projective tractor geometry of $[\nabla]$.  We
obtain an elegant curved analogue of Subsection \ref{model}.

The metric induces canonical trivialisations of the density bundles
$\ce(w)$ and $\ce[w]$, so that we may unambiguously omit all weights.
It thus also induces canonical co-tractor bundle isomorphisms, via the
choices $g\in[g]$ and $\nabla\in[\nabla]$,
\[
\ce_\alpha\cong \ce\oplus\ce_a\oplus\ce;\\ \quad\ce_A\cong\ce_a\oplus
\ce.
\]

\begin{lemma}[\cite{BEG},\cite{GN}]
The tractor $I^\alpha:=
\begin{pmatrix}
1 \\ 0 \\ -\tfrac{1}{n}\widetilde{\J}
\end{pmatrix}$
is parallel.
\end{lemma}

Thus the kernel of $I^\alpha$ (a rank-$(n+1)$ subbundle of the
co-tractor bundle $\ce_\alpha$) is preserved by the tractor
connection.

\begin{proposition}
Define a bundle inclusion $\iota:\ct^*\to \widetilde{\ct}^*$ of the
projective into the conformal co-tractor bundle by, for a projective
co-tractor $U_A=(\mu_a ~ \mid ~ \si)$,
\[
\iota_\alpha{}^A U_A = \left( \tfrac{1}{n}\widetilde{\J} \sigma ~\mid~
\mu_a ~\mid~ \sigma \right).
\]
\begin{enumerate}
\item The image of $\iota$ is the kernel of the parallel tractor
  $I^\alpha$; thus, $\iota_\alpha{}^AI^\alpha=0$.
\item The bundle inclusion $\iota$ is connection-preserving.
\end{enumerate}
\end{proposition}

\begin{proof}
\begin{enumerate}
\item Calculate the pairing of $I^\alpha$ with a tractor in the image
  of this inclusion: for sections $\sigma$ of $\ce$ and $\mu_a$ of
  $\ce_a$, the co-tractor $U_A=(\mu_a ~ \mid ~ \si)$ indeed satisfies
\[
\iota_\alpha{}^A U_AI^\alpha =\tfrac{1}{n}\widetilde{\J} \sigma \cdot
1 -\tfrac{1}{n}\widetilde{\J}\cdot \sigma = 0.
\]

\item Compare the restriction to this sub-bundle of the conformal
  tractor connection (\ref{cstdtracconn}) with the projective tractor
  connection (\ref{tconn}) on its preimage: for sections $\sigma$ of
  $\ce$ and $\mu_a$ of $\ce_a$,
\begin{eqnarray*}
\widetilde{\nabla}_b \left( \tfrac{1}{n}\widetilde{\J} \sigma ~\mid~
\mu_a ~\mid~ \sigma \right) & =&
 \begin{pmatrix}
 \na_b(\tfrac{1}{n}\widetilde{\J} \sigma)
 -\widetilde{\V}_{bc}g^{ca}\cdot \mu_a\\ \na_b\mu_a +g_{ba}\cdot
 \tfrac{1}{n}\widetilde{\J} \sigma +\widetilde{\V}_{ba} \cdot\sigma
 \\ \na_b\sigma -\mu_b
 \end{pmatrix}^T\\
&=&
 \begin{pmatrix}
 \tfrac{1}{n}\widetilde{\J} (\na_b\sigma -\mu_b)\\ \na_b\mu_a
 +2\widetilde{\V}_{ab} \sigma \\ \na_b\sigma -\mu_b
 \end{pmatrix}^T;\\
\na_b (\mu_a ~ \mid ~ \si) &=&(\na_b \mu_a + \V_{ab}\sigma ~ \mid ~
\na_b\sigma - \mu_b);
\end{eqnarray*}
for the former calculation using the Einstein properties that
$\widetilde{\V}_{ab}=\tfrac{1}{n}\widetilde{\J} g_{ab}$ and that
$\widetilde{\J}$ is constant.

Using that $\widetilde{\V}_{ab}=\tfrac{1}{2}\V_{ab}$ (Lemma
\ref{proj-conf-Schouten}) to compare the two right-hand sides, we
obtain, as required, that for $U_A= ( \mu_a ~ \mid ~ \sigma)$,
\[
\widetilde{\na}_b \left(\iota_\alpha{}^A U_A\right)
=\iota_\alpha{}^A\left(\na_b U_A\right).
\]
\end{enumerate}
\end{proof}

\begin{proposition}
The projective tractor sub-metric $h^{AB}$ (Theorem \ref{charth}) is
the pullback under the inclusion $\iota$ of the conformal tractor
metric $\widetilde{h}^{\alpha\beta}$.  That is,
$h^{AB}=\iota_\alpha{}^A\iota_\beta{}^B\widetilde{h}^{\alpha\beta}$.
\end{proposition}

\begin{proof}
Define two projective co-tractors $U_A$, $\underline{U}_B$, by, in the
splitting $\nabla$,
\[
U_A=(\mu_a ~\mid ~ \sigma), \quad \underline{U}_A=( \underline{\mu}_a
~\mid ~ \underline{\sigma}).
\]
The conformal inner product $\widetilde{h}^{\alpha\beta}$
(\ref{cstdtracmetr}) of their images under $\iota$ is then
\[
\widetilde{h}^{\alpha\beta}(\iota_\alpha{}^AU_A)(\iota_\beta{}^B\underline
U_B) = g^{ab}(\mu_a\underline\mu_b) +
\tfrac{2}{n}\widetilde{\J}\sigma\underline\sigma.
\]
By Lemma \ref{proj-conf-Schouten}, this is exactly (\ref{gscale}).
\end{proof}

\end{document}